\documentclass[DIV=13,twoside]{scrartcl}
\usepackage{fixltx2e} 
\usepackage{cmap} 
\usepackage{ifthen}
\usepackage[T1]{fontenc}
\usepackage[utf8]{inputenc}
\usepackage{amsmath}
\usepackage{textcomp} 

\usepackage{amsthm,datetime,scrpage2}

\DeclareMathOperator{\diam}{diam}
\DeclareMathOperator{\dist}{dist}
\DeclareMathOperator{\id}{id}
\DeclareMathOperator{\Fix}{Fix}
\DeclareMathOperator{\const}{const}

\renewcommand{\Im}{\operatorname{Im}}
\renewcommand{\Re}{\operatorname{Re}}
\newcounter{statement}

\newtheorem*{theorem*}{Theorem}
\newtheorem*{maintheorem*}{Main Theorem}
\newtheorem{lemma}[statement]{Lemma}
\newtheorem*{lemma*}{Lemma}
\newtheorem{corollary}[statement]{Corollary}
\newtheorem*{corollary*}{Corollary}

\newtheorem*{conjecture*}{Conjecture}
\theoremstyle{definition}
\newtheorem*{definition}{Definition}
\theoremstyle{remark}
\newtheorem*{remark*}{Remark}

\DeclareUnicodeCharacter{012D}{\u\i}
\DeclareUnicodeCharacter{2254}{\coloneqq}

\newcommand{\DUtitlenote}[1]{}
\AtBeginDocument{\date{\today\ \currenttime}}
\ohead{\headmark}
\ihead{Bounded limit cycles}
\automark[subsection]{section}

\usepackage{amsfonts}
\usepackage{amssymb}
\usepackage{braket}
\usepackage{typearea}
\usepackage{mathtools}

\providecommand{\DUadmonition}[2][class-arg]{%
  \ifcsname DUadmonition#1\endcsname%
    \csname DUadmonition#1\endcsname{#2}%
  \else
    \begin{center}
      \fbox{\parbox{0.9\textwidth}{#2}}
    \end{center}
  \fi
}
\providecommand*{\DUfootnotemark}[3]{%
  \raisebox{1em}{\hypertarget{#1}{}}%
  \hyperlink{#2}{\textsuperscript{#3}}%
}
\providecommand{\DUfootnotetext}[4]{%
  \begingroup%
  \renewcommand{\thefootnote}{%
    \protect\raisebox{1em}{\protect\hypertarget{#1}{}}%
    \protect\hyperlink{#2}{#3}}%
  \footnotetext{#4}%
  \endgroup%
}

\providecommand*{\DUtitle}[2][class-arg]{%
  \ifcsname DUtitle#1\endcsname%
    \csname DUtitle#1\endcsname{#2}%
  \else
    \smallskip\noindent\textbf{#2}\smallskip%
  \fi
}

\providecommand{\DUtopic}[2][class-arg]{%
  \ifcsname DUtopic#1\endcsname%
    \csname DUtopic#1\endcsname{#2}%
  \else
    \begin{quote}#2\end{quote}
  \fi
}

\ifthenelse{\isundefined{\hypersetup}}{
  \usepackage[colorlinks=true,linkcolor=blue,urlcolor=blue,colorlinks=false,unicode]{hyperref}
  \urlstyle{same} 
}{}
\hypersetup{
  pdftitle={Bounded limit cycles of polynomial foliations of ℂP²},
  pdfauthor={Nataliya Goncharuk;Yury Kudryashov}
}

\title{\phantomsection%
  Bounded limit cycles of polynomial foliations of $\mathbb{C}P^{2}$%
  \label{bounded-limit-cycles-of-polynomial-foliations-of-p2}}
\author{Nataliya Goncharuk\thanks{The research was supported by Russian President grant No. MD-2859.2014.1 and by Simons Foundation.} \thanks{Research of both authors was supported by RFBR project 13-01-00969-a, and UNAM-DGAPA-PAPIIT Projects IN 103914 and IN 102413.} \thanks{Address (both authors): National Research University Higher School of Economics, Russia, Moscow, Myasnitskaya str., 20} \and
Yury Kudryashov\thanks{This research carried out in 2014–2015 was supported by “The National Research University ‘Higher School of Economics’ Academic Fund Program” grant No. 14-01-0193.
}}

\begin{document}
\maketitle
\begin{abstract}
In this article we prove in a new way that a generic polynomial vector field in $\mathbb{C}^{2}$ possesses countably many homologically independent limit cycles. The new proof needs no estimates on integrals, provides thinner exceptional set for quadratic vector fields, and provides limit cycles that stay in a bounded domain.
\end{abstract}


\phantomsection\label{contents}
\pdfbookmark[1]{Contents}{contents}
\tableofcontents

\section{Introduction%
  \label{introduction}%
}

Consider a~polynomial differential equation in $\mathbb{C}^{2}$ (with complex time),
\begin{equation}
\begin{aligned}
        {\dot{x}} &= P(x,y),\\
        {\dot{y}} &= Q(x,y),
\end{aligned}
\phantomsection
\label{polynomial-vf}
\end{equation}
where $\max (\deg P, \deg Q) =n$. The splitting of $\mathbb{C}^{2}$ into trajectories of this vector field defines a singular analytic foliation of $\mathbb{C}^{2}$.
Denote by $\mathcal{A}_{n}$ the space of foliations of $\mathbb{C}^{2}$ defined by vector fields \eqref{polynomial-vf} of degree at most $n$ with coprime $P$ and $Q$. Two vector fields define the same foliation if they are proportional, hence $\mathcal{A}_{n}$ is a Zariski open subset of the projective space. $\mathcal{A}_{n}$ is equipped with a natural topology induced from this projective space.

In 1970-s there appeared several results on the properties of generic foliations from $\mathcal{A}_{n}$. In particular, Yu. Ilyashenko \cite{Il78} proved that a generic foliation (more precisely, each foliation from some full Lebesgue measure subset of $\mathcal{A}_{n}$) has an infinite number of limit cycles. Later his theorem was improved by E. Rosales-González, L. Ortiz-Bobadilla and A. Shcherbakov \cite{SRO98}, namely they replaced a full-measure set by an open dense subset.
\begin{definition}
\emph{Limit cycle} on a leaf $L$ of~an~analytic foliation is an element $[\gamma ]$ of the free homotopy group of $L$ such that the holonomy along (any) its representative $\gamma $ is non-identical.
\end{definition}
\begin{definition}
A~set of limit cycles of a~foliation is called \emph{homologically independent}, if for any leaf $L$ all the cycles located on this leaf are linearly independent in $H_{1}(L)$.
\end{definition}
\begin{theorem*}[\cite{Il78}]

For $n\geqslant 2$, there exists a full-measure subset of $\mathcal{A}_{n}$, such that each foliation from this subset possesses an infinite number of homologically independent limit cycles.
\end{theorem*}
\begin{theorem*}[\cite{SRO98}]

For $n\geqslant 3$, there exists an open dense subset of $\mathcal{A}_{n}$, such that each foliation from this subset possesses an infinite number of homologically independent limit cycles.
\end{theorem*}

The proof of the first theorem in \cite{Il78} is rather technical; the proof of the second one in \cite{SRO98} contains about 10 pages of cumbersome estimates of integrals along the limit cycles. The constructed sequence of representatives $\gamma _j$ of required limit cycles $[\gamma _j]$ in both theorems converges to the infinite line.

Our results yield another, less technical proof of these theorems, and our limit cycles are detached from the infinite line. Also, our proof works for $n=2$ for both types of genericity assumptions.
\begin{maintheorem*}[{}]\phantomsection\label{main-theorem}
For $n\geqslant 2$, there exist
\begin{itemize}

\item a full-measure subset $\mathcal{A}^{LC1}_n\subset \mathcal{A}_n$,

\item a complement to a real-analytic subset $\mathcal{A}^{LC2}_n\subset \mathcal{A}_n$,

\end{itemize}

such that each $\mathcal{F}\in \mathcal{A}_n^{LC1}\cup \mathcal{A}_n^{LC2}$ possesses an infinite sequence of limit cycles $[\gamma _j]$ such that:
\newcounter{listcnt0}
\begin{list}{\alph{listcnt0})}
{
\usecounter{listcnt0}
\setlength{\rightmargin}{\leftmargin}
}

\item the cycles are homologically independent;

\item the multipliers of the cycles tend to zero;

\item the cycles are uniformly bounded, i.e., there exists a ball in $\mathbb{C}^{2}$ that includes all representatives $\gamma _j$;

\item there exists a cross-section such that $\gamma _j$ intersect it in a dense subset.
\end{list}

The explicit descriptions of the sets $\mathcal{A}^{LC1}_n$ and~$\mathcal{A}^{LC2}_n$ are given below, in Sections “\hyperref[multiplicative-density]{Multiplicative density}” and “\hyperref[unsolvable-monodromy-group]{Unsolvable monodromy group}”, respectively.
\end{maintheorem*}

The key genericity assumption for $\mathcal{A}^{LC1}_n$ is that the characteristic numbers of two singular points at infinity generate a dense semi-group in $\mathbb{C}/\mathbb{Z}$. The key genericity assumption for $\mathcal{A}^{LC2}_n$ is that the monodromy group at infinity is unsolvable.

Though the exceptional set in the second part is much thinner, we include the first part for two reasons: $\mathcal{A}^{LC2}_n$ does not include $\mathcal{A}^{LC1}_n$, and the first case is technically easier.

Our construction also yields that the infinite number of limit cycles survives in a neighbourhood of $\mathcal{A}^{LC1}_n\cup \mathcal{A}^{LC2}_n$ in the space $\mathcal{B}_{n+1}$ of foliations of $\mathbb{C}P^{2}$ that are given by a~polynomial vector field of degree at~most $n+1$ in any affine chart, see \hyperref[cor-nbhd]{Corollary 6}.

\section{Preliminaries%
  \label{preliminaries}%
}

\subsection{Extension to infinity%
  \label{extension-to-infinity}%
}

Let us extend a polynomial foliation $\mathcal{F}\in \mathcal{A}_n$ given by \eqref{polynomial-vf} to $\mathbb{C}P^{2}$. After changing variables, $u=\frac 1 x, v = \frac y x$, and the time change $d\tau  = -u^{n-1} dt$, the vector field takes the form
\begin{equation}
\begin{aligned}
        {\dot{u}} &= u {\widetilde{P}} (u,v)\\
        {\dot{v}} &= v{\widetilde{P}} (u,v) - {\widetilde{Q}}(u,v)
\end{aligned}
\phantomsection
\label{vf-uv}
\end{equation}
where ${\widetilde{P}} (u,v)= P\left(\frac 1 u, \frac v u\right) u^n$ and ${\widetilde{Q}} (u,v)= Q\left(\frac 1 u,\frac v u\right) u^n$ are two polynomials of degree at most $n$.

Since ${\dot{u}}(0, v)\equiv 0$, the infinite line $\set{u=0}$ is invariant under this vector field. Denote by $h(v)$ the polynomial ${\dot{v}}(0, v)=v{\widetilde{P}} (0,v) - {\widetilde{Q}}(0,v)$. In a generic (\emph{non-dicritical}) case $h(v)\not\equiv 0$; then \eqref{vf-uv} has isolated singular points $a_j\in \set{u=0}$ at the roots of $h$, and $L_\infty  ≔ \set{u=0}\smallsetminus \{a_{1},a_{2},\ldots  \}$ is a leaf of the extension of $\mathcal{F}$ to $\mathbb{C}P^{2}$.

Denote by $\mathcal{A}_{n}'$ the set of foliations $\mathcal{F}\in \mathcal{A}_{n}$ such that $h$ has $n+1$ distinct roots $a_j$, $j=1,\ldots , n+1$. In particular, all these foliations are non-dicritical.

For each $j$, let $\lambda _j$ be the ratio of the eigenvalues of the linearization of \eqref{vf-uv} at $a_j$ (the eigenvalue corresponding to $L_\infty $ is in the denominator). One can show that $\sum \lambda _j=1$, and this is the only relation on $\lambda _j$.

For $\mathcal{F}\in \mathcal{A}_n'$, fix a non-singular point $O\in L_\infty $ and a cross-section $S$ at $O$ given by $v=\const$. Let $\Omega L_\infty $ be the loop space of $(L_\infty , O)$, i.~e., the space of all continuous maps $(S^{1}, pt)\rightarrow (L_\infty , O)$. For a loop $\gamma \in \Omega L_\infty $, denote by $\mathbf{M}_\gamma :(S, O)\rightarrow (S, O)$ the monodromy map along $\gamma $. It is easy to see that $\mathbf{M}_\gamma $ depends only on the class $[\gamma ]\in \pi _{1}(L_\infty , O)$, and the map $\gamma \mapsto \mathbf{M}_\gamma $ reverses the order of multiplication,
\begin{equation*}
\mathbf{M}_{\gamma \gamma '}=\mathbf{M}_{\gamma '}\circ \mathbf{M}_\gamma .
\end{equation*}
The set of all possible monodromy maps $\mathbf{M}_\gamma $, $\gamma \in \Omega L_\infty $, is called the \emph{monodromy pseudogroup} $G=G(\mathcal{F})$.
The word “pseudogroup” means that there is no common domain where all elements of $G$ are defined. However we will follow the tradition and write “monodromy group” instead of “monodromy pseudogroup”.

Choose $n+1$ loops $\gamma _j\in \Omega L_\infty $, $j = 1,2,\ldots ,n+1$, passing around points $a_j$, respectively. We suppose that $\gamma _j$ are simple and intersect only at $O$. Then the pseudogroup $G(\mathcal{F})$ is generated by monodromy maps $\mathbf{M}_j=\mathbf{M}_{\gamma _j}$. It is easy to see that the multipliers $\mu _j=\mathbf{M}_j'(0)$ are equal to $\exp {2 \pi  i \lambda _j}$.

\subsection{Fatou coordinates%
  \label{fatou-coordinates}%
}

The space of germs of analytic parabolic maps $g:(\mathbb{C}, 0)\rightarrow (\mathbb{C}, 0)$, $z\mapsto z+o(z)$, has a~natural filtration by the degree of the leading term of $g(z)-z$. Denote by $A_p$ the set of germs of the form $z\mapsto z+az^{p+1}+o(z^{p+1})$, $a\neq 0$.

In this section we will recall some results on sectorial rectifying charts of parabolic fixed points that will be used in the article. For a~more complete exposition, see, e.g., Chapter IV of \cite{IYa07}.

We start with describing the formal normal forms for quadratic parabolic germs.
\begin{theorem*}[{}]
A~quadratic parabolic germ $f:z\mapsto z+az^{2}+bz^{3}+o(z^{3})$ is formally conjugated to the time-one flow map~$f_\lambda $ of the vector field~$v_\lambda (z) = \frac{z^{2}}{1+\lambda  z}$, where~$\lambda =1-\frac{b}{a^{2}}$. More precisely, there exists a~formal series $H(z)=az+\sum _{k=2}^\infty  h_kz^k$, such that $f_\lambda \circ H=H\circ f$. The series~$H$ is uniquely defined modulo a~formal composition with a~flow map of~$v_\lambda $.
\end{theorem*}

\DUadmonition[note]{
\DUtitle[note]{Note}

It is easy to see that the map $t_\lambda :z\mapsto -\frac 1z+\lambda \log z$ conjugates $f_\lambda $ to the map $t\mapsto t+1$, $t_\lambda (f_\lambda (z))=t_\lambda (z)+1$.
}

We will need the following theorem that describes sectorial rectifying charts for quadratic parabolic germs. Consider the following sectors
\begin{equation*}
\begin{aligned}
        S_{\alpha ,r}^{+}&=\Set{z|\relax |z|<r, |\arg z|<\alpha },&S_{\alpha ,r}^{-}&=\Set{z|\relax |z|<r, |\arg z-\pi |<\alpha }.
\end{aligned}
\end{equation*}\begin{theorem*}[{Sectorial Normalization Theorem}]\phantomsection\label{sectorial-normalization-theorem}
Let $f:z\mapsto z+az^{2}+o(z^{2})$ be a~quadratic parabolic map, let $H(z)=az+\sum _{k=2}^\infty  h_kz^k$ be a~formal series which conjugates $f$ to its formal normal form $f_\lambda $. Then for any $\frac \pi 2<\alpha <\pi $ there exists $r>0$ and a~unique couple of~analytic maps $h^\pm :\frac{1}{a}S^\pm _{\alpha ,r}\rightarrow \mathbb{C}$ with the following properties:
\begin{itemize}

\item $H$ is an~asymptotic series for $h^{-}$ and $h^{+}$: for $N\in \mathbb{N}$, we have $h^\pm (z)=az+\sum _{k=2}^N h_kz^k+o(z^N)$ as $z\rightarrow 0$ inside $\frac 1aS^\pm _{\alpha ,r}$;

\item $h^\pm $ conjugates $f$ to $f_\lambda $: $f_\lambda \circ h^\pm =h^\pm \circ f$.

\end{itemize}
\end{theorem*}

\DUadmonition[note]{
\DUtitle[note]{Note}

For most parabolic germs $f$, $h^{-}\neq h^{+}$. So, the analytic classification of parabolic germs does not coincide with their formal classification. The analytic classification has functional moduli called \emph{Ecalle–Voronin moduli}, namely the restrictions of $(h^{+})^{-1}\circ h^{-}$ to the sectors $\set{z|\relax |z|<r, \pi -\alpha <\arg z<\alpha }$ and $\set{z|\relax |z|<r, -\alpha <\arg z<\pi +\alpha }$ up to a conjugation by a~flow map of $v_\lambda $.
}

\DUadmonition[note]{
\DUtitle[note]{Note}

It is easy to check that the image of $h^\pm $ includes a~sector of the form $S^\pm _{\alpha ', r'}$ for each $\alpha '<\alpha $ and some $r'$, and is included by another sector of the same form. Also, $t_\lambda (S^{-}_{\alpha ',r'})$ includes a~\emph{sector at~infinity}:
\begin{equation}
S_{\beta ,R}^\infty =\Set{\zeta |\relax |\zeta |>R, |\arg \zeta |<\beta }
\phantomsection
\label{sector-at-infinity}
\end{equation}
for each $\beta <\alpha '$ and some $R=R(\alpha ', r', \lambda , \beta )\gg 1$. Thus the image of $\zeta =t_\lambda \circ h^{-}$ includes a~sector at~infinity.
}
\begin{definition}
A \emph{Fatou coordinate} for a~parabolic map $f$ in a~sector $\frac 1a S^{-}_{\alpha ,r}$ is a~coordinate of the form $\zeta =t_\lambda \circ h^{-}$, where $h^{-}$ is given by \hyperref[sectorial-normalization-theorem]{Sectorial Normalization Theorem}. A~Fatou coordinate $\zeta $ conjugates $f$ to the shift $\zeta \mapsto \zeta +1$ in a domain that includes a sector at infinity \eqref{sector-at-infinity}, and is defined uniquely modulo addition of~a~complex number.
\end{definition}

We shall need the following statement.
\begin{lemma}\phantomsection\label{inzetachart}
Let $g$ be a~parabolic map of the form $z\mapsto z+az^2+\ldots $. Let $\zeta $ be a Fatou chart for~$g$ defined in a~sector $\frac 1aS^{-}_{\alpha ,r}$. Let $S^\infty $ be the image of a smaller sector $\frac 1aS^{-}_{\alpha -\varepsilon ,r-\varepsilon }$ under $\zeta $. Let $F:\mathbb{C}\rightarrow \mathbb{C}$ be an~analytic map, $F(0)=0$, defined in the chart $z$. Let ${\tilde{F}}=\zeta \circ F\circ \zeta ^{-1}$ be the corresponding map in the chart $\zeta $.
\setcounter{listcnt0}{0}
\begin{list}{\alph{listcnt0})}
{
\usecounter{listcnt0}
\setlength{\rightmargin}{\leftmargin}
}

\item If $F(z)=kz+o(z)$, then ${\tilde{F}}(\zeta )=k^{-1}\zeta +c+o(1)$ as $\zeta \rightarrow \infty $ inside $S^\infty $.

\item If~$F(z)=z+kz^{p+1}+ o(z^{p+1})$, $p\geqslant 1$, then ${\tilde{F}}(\zeta )=\zeta +(-1)^{p-1}ka^{-p}\frac{1}{\zeta ^{p-1}} +o\left(\frac 1{\zeta ^{p-1}}\right)$ in $\zeta $-chart as $\zeta \rightarrow \infty $ inside $S^\infty $.

\item If $F$ is a~parabolic map, then $\log {\tilde{F}}'(\zeta )=o(F(\zeta )-\zeta )$ as $\zeta \rightarrow \infty $ inside $S^\infty $.
\end{list}
\end{lemma}
\begin{proof}
Recall that $h(z)-az=O(z^{2})$, $z\in \frac 1a S^{-}_{\alpha ,r}$. The Cauchy estimates imply that $h'(z)=a+O(z)$ in $\frac 1a S_{\alpha -\varepsilon ,r-\varepsilon }$.

Let us prove b). Note that
\begin{align*}
F(h^{-1}(w))&=h^{-1}(w)+k(h^{-1}(w))^{p+1}+o((h^{-1}(w))^{p+1})\\
&=h^{-1}(w)+ka^{-p-1}w^{p+1}+o(w^{p+1}),
\end{align*}
hence
\begin{equation*}
(h\circ F\circ h^{-1})(w)=w+\int _{h^{-1}(w)}^{F(h^{-1}(w))}a+O(z)\,dz=w+ka^{-p}w^{p+1}+o(w^{p+1}).
\end{equation*}
Similarly, for $\zeta =t_\lambda (w)$ we have
\begin{align*}
{\tilde{F}}(\zeta )&=\zeta +\int _w^{w+ka^{-p}w^{p+1}+o(w^{p+1})}t_\lambda '(\omega )\,d\omega \\
&=\zeta +\int _w^{w+ka^{-p}w^{p+1}+o(w^{p+1})}\frac 1{\omega ^{2}}+\frac \lambda \omega \,d\omega \\
&=\zeta +ka^{-p}w^{p-1}+o(w^{p-1})\\
&=\zeta +(-1)^{p-1}ka^{-p}\frac{1}{\zeta ^{p-1}} +o\left(\frac 1{\zeta ^{p-1}}\right)
\end{align*}
Assertion a) can be proved in the same way. Finally, the last assertion follows from Assertion b) in $\zeta \left(\frac 1aS^{-}_{\alpha -\frac \varepsilon 2,r-\frac \varepsilon 2}\right)$ and Cauchy estimates.
\end{proof}

\subsection{Unsolvability of the monodromy group%
  \label{unsolvability-of-the-monodromy-group}%
}

In \cite{Shch84}, A. Shcherbakov proved that for a generic foliation $\mathcal{F}\in \mathcal{A}_{n}$, the monodromy group is unsolvable. It turns out that a~group of~germs $(\mathbb{C}, 0)\rightarrow (\mathbb{C}, 0)$ is unsolvable if and only if it contains a~couple of~commutators that do not commute with each other. This follows from the next lemma.
\begin{lemma}\phantomsection\label{lemma-parabolic-commutator}
Let $f(z)=z+az^{p+1}+o(z^{p+1})$ and $g(z)=z+bz^{q+1}+o(z^{q+1})$ be two parabolic germs. Then
\begin{equation*}
[f, g](z)≔(f\circ g\circ f^{-1}\circ g^{-1})(z)=z+ab(p-q)z^{p+q+1}+o(z^{p+q+1}).
\end{equation*}
In particular, if $p\neq q$, $a\neq 0$, $b\neq 0$, then $[f, g]\in A_{p+q}$.
\end{lemma}
\begin{corollary*}[{}]
If a~group~$G$ of~germs $(\mathbb{C}, 0)\rightarrow (\mathbb{C}, 0)$ contains two parabolic germs $g_{1}\in A_p$, $g_{2}\in A_q$ with $p\neq q$, then $G$ is unsolvable.
\end{corollary*}
\begin{proof}
Indeed, none of the commutators $g_{3}=[g_{1}, g_{2}]\in A_{p+q}$, $g_{4}=[g_{3}, g_{2}]\in A_{p+2q}$, … can be the identity map.
\end{proof}

The main result of this section is the following lemma.
\begin{lemma}\phantomsection\label{unsolvable-details}
There exists an~open dense subset $\mathcal{A}_{n}$ such that for each foliation from this subset
\begin{equation*}
\forall i\neq j\quad[\mathbf{M}_i,\mathbf{M}_j]\in A_{1},\quad [\mathbf{M}_i^{-1},\mathbf{M}_j]\in A_{1},\quad [[\mathbf{M}_i,\mathbf{M}_j],[\mathbf{M}_i^{-1},\mathbf{M}_j]]\in A_{3}.
\end{equation*}\end{lemma}

This lemma is immediately implied by the following two statements.
\begin{lemma}\phantomsection\label{scherbakov-unsolvable}
There exists a~real analytic subset $\mathcal{E}\subset \mathcal{A}_n$ of~positive codimension such that for $\mathcal{F}\notin \mathcal{E}$
\begin{itemize}

\item all commutators $[\mathbf{M}_i,\mathbf{M}_j]$, $i\neq j$, belong to $A_{1}$;

\item all numbers $\dfrac{S(\mathbf{M}_i)(0)}{\mathbf{M}_i'(0)^{2}-1}$, $i=1,\ldots ,n+1$, are different.

\end{itemize}
\end{lemma}

Here and below $S(f)$ is the \href{https://en.wikipedia.org/wiki/Schwarzian_derivative}{Schwarzian derivative} of $f$,
\begin{equation*}
S(f)(z)=\frac{f'''(z)}{f'(z)}-\frac 32\left(\frac{f''(z)}{f'(z)}\right)^{2}.
\end{equation*}
In \cite{Shch84}, Shcherbakov proved that for a~generic foliation, \emph{at least one} commutator $[\mathbf{M}_i,\mathbf{M}_j]$ belongs to $A_{1}$, see Section 6.3 of \cite{Shch06}. But it is easy to extend this result to \emph{all} pairs $i\neq j$ using analytic continuation along loops in~$\mathcal{A}_{n}$ that permute the singular points. The second part is proved in the same article but not explicitly stated, so one needs to go through the proof of Theorem 9 in \cite{Shch06} to verify that the assertion of the corollary after Lemma~5 is~the only property of $\mathcal{F}$ used in the proof.

Similar results were achieved in \cite{BLL97,N94}.
\begin{lemma*}[{}]
Consider two hyperbolic germs $f, g$ such that
\begin{itemize}

\item $f'(0)^{2}\neq 1$, $g'(0)^{2}\neq 1$;

\item $[f, g]\in A_{1}$, i.e., $[f, g]''(0)\neq 0$;

\item $\displaystyle\frac{S(f)(0)}{f'(0)^{2}-1}\neq \frac{S(g)(0)}{g'(0)^{2}-1}$.

\end{itemize}

Then $[f, g]$ does not commute with $[f^{-1}, g]$; moreover, $[[f, g], [f^{-1}, g]]\in A_{3}$.
\end{lemma*}

This lemma is motivated by Proposition~7 in \cite{Shch84} (which coincides with the corollary after Lemma~5 in \cite{Shch06}) but provides an explicit pair of commutators that do not commute.
\begin{proof}
One can verify that
\begin{equation*}
S([f,g])(0)=\left[\left(\frac{S(f)}{(f')^{2}-1}-\frac{S(g)}{(g')^{2}-1}\right)\left(1-\frac1{(f')^{2}}\right)\left(1-\frac{1}{(g')^{2}}\right)\right]_{z=0}
\end{equation*}
thus
\begin{equation*}
\frac{S([f^{-1},g])(0)}{S([f,g])(0)}=-f'(0)^{2}
\end{equation*}
On the other hand,
\begin{equation}
\frac{[f^{-1},g]''(0)}{[f,g]''(0)}=-f'(0).
\phantomsection
\label{commutators-d2}
\end{equation}
Two last equalities prove that
\begin{equation*}
\frac{S([f^{-1},g](0)}{[f^{-1},g]''(0)}=f'(0)\frac{S([f,g](0)}{[f,g]''(0)}
\end{equation*}
The assertions of the lemma imply that this is not zero, thus
\begin{equation*}
\frac{S([f^{-1},g](0)}{[f^{-1},g]''(0)}\neq \frac{S([f,g](0)}{[f,g]''(0)}.
\end{equation*}
Finally, expanding $[f, g]\circ [f^{-1}, g]-[f^{-1},g]\circ [f,g]$ up to the fourth order, one can check that the above inequality is equivalent to
\begin{equation*}
([f, g]\circ [f^{-1}, g]-[f^{-1},g]\circ [f,g])^{(4)}(0)\neq 0,
\end{equation*}
hence $[[f,g],[f^{-1},g]]\in A_{3}$.
\end{proof}

\section{Plan of the proof of \hyperref[main-theorem]{Main Theorem}%
  \label{plan-of-the-proof-of-main-theorem}%
}

We will construct the limit cycles as the lifts of loops in the infinite line.
Note that if the monodromy map $\mathbf{M}_{k_{1}}\mathbf{M}_{k_{2}}\ldots  \mathbf{M}_{k_l}:S\rightarrow S$ has a fixed point $p\neq 0$, then the corresponding loop $\gamma =\gamma _{k_l}\gamma _{k_{l-1}}\ldots  \gamma _{k_{1}}$ lifts to a limit cycle $c$ starting from $p$;  the projection of  $c$ to the infinite line is~$\gamma $.

We proceed in two steps. First we construct contracting monodromy maps that satisfy \hyperref[inclusion]{\textbf{inclusion}}, \hyperref[contraction]{\textbf{contraction}} and \hyperref[covering]{\textbf{covering}} assumptions formulated below. This is done in a~different way for two types of genericity assumptions, see Section “\hyperref[construction-of-contracting-maps]{Construction of contracting maps}”. Then we use the maps constructed on the first step to obtain limit cycles that satisfy the assertions of \hyperref[main-theorem]{Main Theorem}. On this step we use no information about the foliation except for existence of maps with prescribed properties, see Section “\hyperref[construction-of-limit-cycles]{Construction of limit cycles}”.

\subsection{Step 1: contracting maps%
  \label{step-1-contracting-maps}%
}

We shall find two topological discs $\Delta ^{-}\subset \Delta ^{+}\subset S$ in the cross-section, $0\notin \Delta ^{+}$, an analytic chart $\zeta $ in $\Delta ^{+}$ and a tuple of monodromy maps $f_j$ with the following properties. Each $f_j$ is a composition of standard generators $\mathbf{M}_k$ of the monodromy group at infinity. For any splitting of this composition into two parts $f_j = f_j^{(t)}\circ  f_j^{(h)}$, we will say that $f_j^{(t)}$ is a \emph{tail} of $f_j$ and $f_j^{(h)}$ is its \emph{head}.
\begin{description}
\item[{Inclusion:}] \leavevmode 
\phantomsection\label{inclusion}
$f_j(\Delta ^{+})\subset \Delta ^{+}$ for any $j$.

\item[{Contraction:}] \leavevmode 
\phantomsection\label{contraction}
All compositions of the form $f_i^{(t)}\circ  f_j^{(h)}$, $f_i^{(h)}\neq \id$, $f_j^{(h)}\neq \id$, contract in $(f_j^{(h)})^{-1}\circ f_i^{(h)}(\Delta ^{+})\cap \Delta ^{+}$ with respect to the chart $\zeta $. In particular, all $f_j$ contract in $\Delta ^{+}$.

\item[{Covering:}] \leavevmode 
\phantomsection\label{covering}
Images of $\Delta ^{-}$ under $f_j$ cover $\Delta ^{-}$.

\end{description}

We will also suppose that the compositions $f_j$ do not contain identical subcompositions, otherwise we remove them. Obviously this does not break any of other requirements on $f_j$.

\subsection{Step 2: limit cycles%
  \label{step-2-limit-cycles}%
}

Here we use the maps $f_j$ to construct infinitely many homologically independent limit cycles. We will use not a particular construction of $f_j$, but only the assumptions \hyperref[inclusion]{\textbf{inclusion}},  \hyperref[contraction]{\textbf{contraction}} and \hyperref[covering]{\textbf{covering}}, so the arguments work for both sets $A^{LC1}_n, A^{LC2}_n$.  The main motivation is the following lemma\DUfootnotemark{id16}{id17}{1}.
\DUfootnotetext{id17}{id16}{1}{%
Some people attribute this statement to Hutchinson \cite{H81}, but we failed to find exactly this statement in this article.
}
\begin{lemma}\phantomsection\label{hutchinson}
Under assumptions above, for an open $U\subset \Delta ^{-}$ and $\varepsilon >0$  there exists a word $w = j_{1} \ldots  j_N$ such that the monodromy map $f_w=f_{j_{1}}\circ f_{j_{2}}\circ \ldots \circ  f_{j_N}$ satisfies $f_w(\Delta ^{+})\subset U$ and $|f'_w|<\varepsilon $ in $\Delta ^{+}$.
\end{lemma}

\DUadmonition[note]{
\DUtitle[note]{Note}

We will use this lemma only  for $\varepsilon <1$. In this case, the map $f_w$ obviously has a fixed point in $U$. It corresponds to a limit cycle with arbitrarily small multiplier which passes through~$U$.
}

This lemma enables us to prove assertions b)–d) of \hyperref[main-theorem]{Main Theorem}. The proof of homological independence is more complicated.
\begin{proof}
Take a point $p\in U\subset \Delta ^{-}$. Due to the \hyperref[covering]{\textbf{covering}} assumption, there exists an index $j_{1}$ such that $p \in  f_{j_{1}}(\Delta ^{-})$. Now, take the preimage $f_{j_{1}} ^{-1}(p)\subset \Delta ^{-}$, and repeat the arguments; we obtain a map $f_{j_{2}}$ such that $p\in f_{j_{1}}( f_{j_{2}} (\Delta ^{-}))$.

Repeating the procedure, we get a word $w=j_{1}\, j_{2}\, \ldots  \,j_N$ such that $f_w (\Delta ^{-}) = f_{j_{1}}\circ f_{j_{2}}\circ \cdots \circ  f_{j_N} (\Delta ^{-})$ contains $p$. The diameter of the image $f_w(\Delta ^{+})$ tends to zero as $N$ tends to infinity, since all maps $f_j$ \hyperref[contraction]{contract} in $\Delta ^{+}$. So, if $N$ is large enough, the fact that $p \in  f_w (\Delta ^{+})$ would imply $f_w (\Delta ^{+})\subset U$ and $|f'_w|<\varepsilon $ in the whole $\Delta ^{+}$.
\end{proof}

\subsection{A neighborhood in $\mathcal{B}_n$%
  \label{a-neighborhood-in-n}%
}

Note that assumptions \hyperref[inclusion]{\textbf{inclusion}}, \hyperref[contraction]{\textbf{contraction}} and \hyperref[covering]{\textbf{covering}} are robust in the following sense. Considers a~foliation $\mathcal{F}\in \mathcal{A}_{n}$ that possesses a~tuple of~monodromy maps that satisfy these assumptions. Then there exists a~bidisc $D\subset \mathbb{C}^{2}$ and $\varepsilon >0$ such that any foliation $\mathcal{F}'$ of~$D$ which is $\varepsilon $-close to $\mathcal{F}$ in $D$ possesses monodromy maps that satisfy these assumptions. Since Step 2 relies only on these properties, such foliation $\mathcal{F}'$ satisfies assertions of Main Theorem. In particular, we have the following corollary.
\begin{corollary}\phantomsection\label{cor-nbhd}
Any foliation from some open neighborhood $\mathcal{U}$, $\mathcal{A}^{LC1,2}_{n}\subset \mathcal{U}\subset \mathcal{B}_{n+1}$, possesses an infinite number of limit cycles satisfying assertions a)–d) of the Main theorem.
\end{corollary}

\section{Construction of contracting maps%
  \label{construction-of-contracting-maps}%
}

\subsection{Multiplicative density%
  \label{multiplicative-density}%
}

We put the following genericity assumptions on the foliation:
\begin{itemize}

\item 
\phantomsection\label{density-condition}
the characteristic numbers of two of the singular points at infinity (say, $\lambda _{1}$ and $\lambda _{2}$) generate a dense subgroup in $\mathbb{C}/\mathbb{Z}$;

\item the corresponding monodromy maps $\mathbf{M}_{1}, \mathbf{M}_{2}$ do not commute.

\end{itemize}

Each condition defines a~full measure set. For the former condition it is clear, and for the latter one see \hyperref[scherbakov-unsolvable]{Lemma 4}.

After a holomorphic coordinate change, we may and will assume that the map $\mathbf{M}_{1}$ is linear. If this map expands, let us replace it with its inverse. Then $\Im \lambda _{1}>0$. Let us pass to the chart $\zeta =\frac{\log z}{2\pi i}$, $\zeta \in \mathbb{C}/\mathbb{Z}$. In this chart, points with large $\Im \zeta $ correspond to points $z$ close to the origin.

Let ${\tilde{\mathbf{M}}}_{1}:\zeta \mapsto \zeta +\lambda _{1}$ and ${\tilde{\mathbf{M}}}_{2}$ be the maps $\mathbf{M}_{1}$ and $\mathbf{M}_{2}$ written in the chart $\zeta $. These maps are defined for sufficiently large $\Im \zeta $, and
\begin{align*}
{\tilde{\mathbf{M}}}_{2}(\zeta )&=\zeta +\lambda _{2}+o(1),\\
{\tilde{\mathbf{M}}}_{2}'(\zeta )&=1+o(1)
\end{align*}
as $\Im \zeta \rightarrow \infty $. Since $\mathbf{M}_{1}$ does not commute with $\mathbf{M}_{2}$, the map ${\tilde{\mathbf{M}}}_{2}$ is not linear, hence ${\tilde{\mathbf{M}}}_{2}'$ is not identically one. Let $\Delta ^{+}$ be a~small closed disc such that either ${\tilde{\mathbf{M}}}_{2}$ or its inverse uniformly contracts in $\Delta ^{+}$. Without loss of generality we can and shall assume that it is ${\tilde{\mathbf{M}}}_{2}$,
\begin{equation}
\max_{\zeta \in \Delta ^{+}}|{\tilde{\mathbf{M}}}_{2}'(\zeta )|<1.
\phantomsection
\label{m2-contracts}
\end{equation}
Next, let $\zeta _{0}$ be the center of $\Delta ^{+}$, put $T=\zeta _{0}-{\tilde{\mathbf{M}}}_{2}(\zeta _{0})$. Note that ${\tilde{\mathbf{M}}}_{2}(\Delta ^{+})+T\Subset \Delta ^{+}$. Choose a~much smaller disc $\Delta ^{-}\subset \Delta ^{+}$ with the same center,
\begin{equation}
\diam(\Delta ^{-})<\dist({\tilde{\mathbf{M}}}_{2}(\partial \Delta ^{+})+T, \partial \Delta ^{+}).
\phantomsection
\label{delta-minus-small}
\end{equation}
Choose a~tuple of vectors $T_j=k_j\lambda _{1}+l_j\lambda _{2}\in \mathbb{C}/\mathbb{Z}$ such that
\begin{equation}
\Delta ^{-}\subset \bigcup _j ({\tilde{\mathbf{M}}}_{2}(\Delta ^{-})+T_j)\subset \bigcup _j ({\tilde{\mathbf{M}}}_{2}(\Delta ^{+})+T_j)\subset \Delta ^{+}.
\phantomsection
\label{tj}
\end{equation}
Due to \eqref{delta-minus-small}, it is enough to take $T_j$ such that ${\tilde{\mathbf{M}}}_{2}(\Delta ^{-})+T_j$ cover $\Delta ^{-}$ and $|T-T_j|<\diam(\Delta ^{-})$. Due to the \hyperref[density-condition]{density condition}, these $T_j$ can be chosen of the form $T_j=k_j\lambda _{1}+l_j\lambda _{2}$.

Now, let us choose $f_j$ so that they approximate the maps $T_j\circ {\tilde{\mathbf{M}}}_{2}$ in $\Delta ^{+}$. Note that the map ${\tilde{\mathbf{M}}}_{2}^{l_j}\circ {\tilde{\mathbf{M}}}_{1}^{k_j}$ approximates the shift $\zeta \mapsto \zeta +T_j$ for large $\Im \zeta $, hence for $N$ large enough, the map ${\tilde{\mathbf{M}}}_{1}^{-N}\circ {\tilde{\mathbf{M}}}_{2}^{l_j}\circ {\tilde{\mathbf{M}}}_{1}^{k_j+N}$ is very close to the shift by $T_j$ in  $C^{1}({\tilde{\mathbf{M}}}_{2}(\Delta ^{+}))$. Therefore, we can take the maps
\begin{equation*}
f_j=(\mathbf{M}_{1}^{-N}\circ \mathbf{M}_{2}^{l_j}\circ \mathbf{M}_{1}^{k_j+N})\circ \mathbf{M}_{2}.
\end{equation*}
Let ${\tilde{f}}_j$ be the map $f_j$ written in the chart~$\zeta $.

For sufficiently large $N$, these maps satisfy \hyperref[inclusion]{\textbf{inclusion}} and \hyperref[covering]{\textbf{covering}} assumptions from the \hyperref[plan-of-the-proof-of-main-theorem]{plan of the proof}. Obviously, each~${\tilde{f}}_j$ contracts in~$\Delta ^{+}$.  Now, we have to prove \hyperref[contraction]{\textbf{contraction}} assumption, i.e. that for $N$ large enough all compositions ${\tilde{f}}_i^{(t)}\circ {\tilde{f}}_j^{(h)}$ contract. Recall that ${\tilde{f}}_j^{(h)}\neq \id$, ${\tilde{f}}_i^{(t)}\neq {\tilde{f}}_i$.

Since the map ${\tilde{\mathbf{M}}}_{1}^{-N}\circ {\tilde{\mathbf{M}}}_{2}^{l_j}\circ {\tilde{\mathbf{M}}}_{1}^{k_j+N}$ and its heads approximate shifts in $C^{1}(\Delta ^{+})$, for $N$ large enough the derivative of ${\tilde{f}}_j^{(h)}$ is close to the set ${\tilde{\mathbf{M}}}_{2}'(\Delta ^{+})$. On the other hand, the derivative of ${\tilde{f}}_i^{(t)}$ on the set ${\tilde{f}}_i^{(h)}(\Delta ^{+})$  is arbitrarily close to one. Thus we can make the derivative of the composition ${\tilde{f}}_i^{(t)}\circ {\tilde{f}}_j^{(h)}$  on the set $\Delta ^{+}\cap \left(({\tilde{f}}_j^{(h)})^{-1}\circ {\tilde{f}}_i^{(h)}(\Delta ^{+})\right)$ arbitrarily close to ${\tilde{\mathbf{M}}}_{2}'(\Delta ^{+})$. Now \eqref{m2-contracts} yields \hyperref[contraction]{\textbf{contraction}} assumption.

\subsection{Unsolvable monodromy group%
  \label{unsolvable-monodromy-group}%
}

In this case, the construction is similar, but instead of the logarithmic chart we use the \hyperref[fatou-coordinates]{Fatou chart} for one of the parabolic monodromy maps, and there are more technical difficulties.

\subsubsection{Genericity assumptions and preliminary considerations%
  \label{genericity-assumptions-and-preliminary-considerations}%
}

Let $\mathcal{A}_{n}^{LC2}\subset \mathcal{A}_{n}'$ be the set of polynomial foliations such that
\begin{itemize}

\item $g_{1}≔[\mathbf{M}_{1},\mathbf{M}_{2}]\in A_{1}$, $g_{2}≔[\mathbf{M}_{1}^{-1},\mathbf{M}_{2}]\in A_{1}$;

\item $g_{3}≔[g_{2},g_{1}]$ is not the identity map\DUfootnotemark{id19}{id20}{2};

\item $\mu _{1}\notin \mathbb{R}$;

\item the numbers $1$, $\mu _{1}$, $\mu _{1}^{-1}$, $\mu _{2}^{-1}$, $\mu _{1}^{-1}\mu _{2}^{-1}$, $\mu _{1}\mu _{2}^{-1}$ are all different.

\end{itemize}
\DUfootnotetext{id20}{id19}{2}{%
\hyperref[unsolvable-details]{Lemma 3} implies that for a~generic~$\mathcal{F}$ we have $g_{3}\in A_{3}$, but our construction works for a~slightly broader set of~foliations.
}

Due to \hyperref[unsolvable-details]{Lemma 3} and the fact that $\sum \lambda _i=1$ is the only relation on $\lambda _i$, the complement $\mathcal{A}_{n}\smallsetminus \mathcal{A}_{n}^{LC2}$ is a~real analytic subset of $\mathcal{A}_{n}$.

Consider a~foliation~$\mathcal{F}$ from this set. Put $g_{4}≔[g_{3},g_{2}]$. Due to \hyperref[lemma-parabolic-commutator]{Lemma 2}, $g_{4}\neq \id$.

Let $\zeta $ be a \hyperref[fatou-coordinates]{Fatou chart} for $g_{1}$ in the negative sector.

\DUtopic[]{
\DUtitle[title]{Convention}

In this Section, tilde above means that a~map is written in the chart~$\zeta $. In particular, ${\tilde{g}}_{1}(\zeta )= \zeta +1$.
}

\hyperref[lemma-parabolic-commutator]{Lemma 2} and Item b) of \hyperref[inzetachart]{Lemma 1} imply
\begin{align*}
g_{1}(z)&=z+\frac{g_{1}''(0)}{2}z^{2}+o(z^{2})&{\tilde{g}}_{1}(\zeta )&=\zeta +1+o(1)\\
g_{2}(z)&=z+\frac{g_{2}''(0)}{2}z^{2}+o(z^{2})&{\tilde{g}}_{2}(\zeta )&=\zeta +\frac{g_{2}''(0)}{g_{1}''(0)}+o(1)\\
g_{3}(z)&=z+az^{p+1}+o(z^{p+1})&{\tilde{g}}_{3}(\zeta )&=\zeta -\frac{(-2)^pa}{g_{1}''(0)^p}\zeta ^{1-p}+o(\zeta ^{1-p})\\
g_{4}(z)&=z+\frac{a(p-1)g_{2}''(0)}2z^{p+2}+o(z^{p+2})&{\tilde{g}}_{4}(\zeta )&=\zeta +\frac{(-2)^pa(p-1)g_{2}''(0)}{g_{1}''(0)^{p+1}}\zeta ^{-p}+o(\zeta ^{-p}),
\end{align*}
where $a\in \mathbb{C}\smallsetminus \set{0}$. Put ${\tilde{a}}=-\frac{(-2)^pa}{g_{1}''(0)^p}$, ${\tilde{b}}=\frac{(-2)^pa(p-1)g_{2}''(0)}{g_{1}''(0)^{p+1}}$. Due to \eqref{commutators-d2}, $g_{2}''(0)=-\mu _{1}g_{1}''(0)$, hence
\begin{align*}
{\tilde{g}}_{1}(\zeta )&=\zeta +1+o(1)\\
{\tilde{g}}_{2}(\zeta )&=\zeta -\mu _{1}+o(1)\\
{\tilde{g}}_{3}(\zeta )&=\zeta +\frac{{\tilde{a}}}{\zeta ^{p-1}}+o\left(\frac{1}{\zeta ^{p-1}}\right)\\
{\tilde{g}}_{4}(\zeta )&=\zeta +\frac{{\tilde{b}}}{\zeta ^p}+o\left(\frac{1}{\zeta ^p}\right).
\end{align*}
Since $\frac{{\tilde{b}}}{{\tilde{a}}}=-\frac{(p-1)g_{2}''(0)}{g_{1}''(0)}=(p-1)\mu _{1}\notin \mathbb{R}$, each number $T\in \mathbb{C}$ can be represented as
\begin{equation}
T=\xi (T){\tilde{a}}+\eta (T){\tilde{b}}.
\phantomsection
\label{xi-eta}
\end{equation}
We will construct $f_j$ as compositions of the maps $g_{1}, g_{2}, g_{3}, g_{4}$.

\subsubsection{Construction of $f_j$%
  \label{construction-of-f-j}%
}

Let $\Delta ^{+}$ be a small disc that we shall choose later. Now we just say that ${\tilde{g}}_{2}^\pm \in \set{{\tilde{g}}_{2},{\tilde{g}}_{2}^{-1}}$ contracts in $\Delta ^{+}$,
\begin{equation}
\max_{\zeta \in \Delta ^{+}}|({\tilde{g}}_{2}^\pm )'(\zeta )|=q<1.
\phantomsection
\label{g2-contracts}
\end{equation}
Since ${\tilde{g}}_{2}(\zeta )-\zeta \rightarrow -\mu _{1}$ as $\zeta \rightarrow \infty $, we may and will assume that $|{\tilde{g}}_{2}^\pm (\zeta )-\zeta |<|\mu _{1}|+1$ for $\zeta \in \Delta ^{+}$.

As in the \hyperref[multiplicative-density]{previous case}, take a~disc $\Delta ^{-}\subset \Delta ^{+}$ and a~tuple of vectors $T_j\in \mathbb{C}$ such that
\begin{equation}
\Delta ^{-}\subset \bigcup _j ({\tilde{g}}_{2}^\pm (\Delta ^{-})+T_j)\subset \bigcup _j ({\tilde{g}}_{2}^\pm (\Delta ^{+})+T_j)\subset \Delta ^{+}.
\phantomsection
\label{g2-tj}
\end{equation}
It is easy to see that the right inclusion implies $|T_j+\mu _{1}|<2$ or $|T_j-\mu _{1}|<2$, hence $|T_j|<|\mu _{1}|+2$. Put $\xi _j=\xi (T_j)$, $\eta _j=\eta (T_j)$, see \eqref{xi-eta}.

Similarly to the \hyperref[multiplicative-density]{previous case}, we will choose the compositions ${\tilde{f}}_j$ so that they will approximate the maps $\zeta \mapsto {\tilde{g}}_{2}^\pm (\zeta )+T_j$ in $C^{1}(\Delta ^{+})$. It turns out that we can use the compositions ${\tilde{f}}_j={\tilde{F}}_j\circ {\tilde{g}}_{2}^\pm $, where
\begin{equation}
{\tilde{F}}_j ≔ {\tilde{g}}_{1}^{-N}\circ {\tilde{g}}_{3}^{k_j}\circ {\tilde{g}}_{4}^{l_j}\circ {\tilde{g}}_{1}^{N},\quad k_j=[N^{p-1}\xi _j],\quad l_j=[N^p\eta _j].
\phantomsection
\label{f-j}
\end{equation}
Here $N$ is a large number that we will choose later.

Let us prove that ${\tilde{F}}_j$ approximate the translations $\zeta \mapsto \zeta +T_j$ in $C^{1}({\tilde{g}}_{2}^\pm (\Delta ^{+}))$.
\begin{lemma}\phantomsection\label{f-j-head}
For $N$ large enough, each head ${\tilde{F}}_j^{(h)}$ such that $\left(F_j^{(h)}\right)'(0)=1$ is close to a~translation $\zeta \mapsto \zeta +T$ in $C^{1}({\tilde{g}}_{2}^\pm (\Delta ^{+}))$. Moreover, $\Re T>-|\Re \mu _{1}|-2$ and $|\Im T|$ is bounded by a~number that does not depend on $\Delta ^{+}$.

In particular,
\setcounter{listcnt0}{0}
\begin{list}{\alph{listcnt0})}
{
\usecounter{listcnt0}
\setlength{\rightmargin}{\leftmargin}
}

\item ${\tilde{g}}_{1}^n$ is the translation by $n$;

\item ${\tilde{g}}_{4}^{n}\circ {\tilde{g}}_{1}^N$, $|n|\leqslant |k_j|$ is close to the translation by~$N+\frac{n{\tilde{b}}}{N^p}$;

\item ${\tilde{g}}_{3}^{n}\circ {\tilde{g}}_{4}^{l_j}\circ {\tilde{g}}_{1}^N$, $|n|\leqslant |k_j|$ is close to the translation by~$N+\frac{n{\tilde{a}}}{N^{p-1}}+{\tilde{b}}\eta _j$;

\item ${\tilde{g}}_{1}^{-n}\circ {\tilde{g}}_{3}^{k_j}\circ {\tilde{g}}_{4}^{l_j}\circ {\tilde{g}}_{1}^N$, $0\leqslant n\leqslant N$, is close to the translation by $T_j+N-n$.
\end{list}
\end{lemma}
\begin{proof}
We shall prove this lemma only for $\xi _j>0$ and $\eta _j>0$. For other cases, it is enough to replace ${\tilde{g}}_{3}$ and (or) ${\tilde{g}}_{4}$ by its inverse map.

Let us prove the assertions of the lemma for ${\tilde{F}}_j^{(h)}={\tilde{g}}_{4}^{l_j}\circ {\tilde{g}}_{1}^N$. Recall that ${\tilde{g}}_{4}(\zeta )-\zeta ={\tilde{b}}\zeta ^{-p}+o(\zeta ^{-p})$, hence
\begin{equation*}
{\tilde{F}}_j^{(h)}(\zeta )={\tilde{g}}_{4}^{l_j}(\zeta +N)=\zeta +N+\frac{{\tilde{b}}l_j}{(\zeta +N)^p}+o(1)=\zeta +N+\frac{{\tilde{b}}[N^p\eta _j]}{N^p}+o(1)=\zeta +N+{\tilde{b}}\eta _j+o(1)
\end{equation*}
as $N\rightarrow \infty $, $\zeta \in {\tilde{g}}_{2}^\pm (\Delta ^{+})$. Therefore, ${\tilde{F}}_j^{(h)}$ is $C^{0}$-close to the translation by $N+{\tilde{b}}\eta _j$.

Let us prove that the derivative of ${\tilde{F}}_j^{(h)}$ is close to one. Due to Item c) of \hyperref[inzetachart]{Lemma 1}, $\log {\tilde{g}}_{4}'(\zeta )=o({\tilde{g}}_{4}(\zeta )-\zeta )=o(\zeta ^{-p})$ as $\zeta \rightarrow \infty $, hence
\begin{equation*}
\log\left({\tilde{F}}_j^{(h)}\right)'(\zeta )=\sum _{k=0}^{l_j-1}\log {\tilde{g}}_{4}'({\tilde{g}}_{4}^k(\zeta +N))\leqslant N^p\eta _j\log {\tilde{g}}_{4}'(N+O(1))=o(1)
\end{equation*}
as $N\rightarrow \infty $, $\zeta \in \Delta ^{+}$. Thus
\begin{equation*}
\left({\tilde{F}}_j^{(h)}\right)'(\zeta )=1+o(1).
\end{equation*}
Finally, in this case ${\tilde{F}}_j^{(h)}$ is $C^{1}$-close to~the~translation by $N+{\tilde{b}}\eta _j$.

All particular cases listed in the statement of the lemma can be proved in the same way. Also, the estimate $|T_j|<|\mu _{1}|+2$ yields a~uniform estimate on the imaginary parts of the translation vectors.

Consider a~head~${\tilde{F}}_j^{(h)}$, $\left(F_j^{(h)}\right)'(0)=1$, not listed explicitly in the statement of the lemma. Since $g_{1}$ and $g_{2}$ have no heads $g$ with $g'(0)=1$, ${\tilde{F}}_j^{(h)}$ differs from a~head of type b) or c) by a~composition with a~head ${\tilde{g}}$ either of ${\tilde{g}}_{3}^{\pm 1}$, or of ${\tilde{g}}_{4}^{\pm 1}$ such that $g'(0)=1$. Since ${\tilde{g}}$ is applied to points $\zeta $ with $\Re \zeta =N+O(1)$, it can be made arbitrarily $C^{1}$-close to its “translational part” $\zeta \mapsto \zeta +T$, $T=\lim_{\zeta \rightarrow \infty }{\tilde{g}}(\zeta )-\zeta $. Thus ${\tilde{F}}_j^{(h)}$ is close to a~translation $\zeta \mapsto \zeta +T'$ with bounded $\Im T'$ as well.
\end{proof}

\subsubsection{Choice of $\Delta ^{+}$%
  \label{choice-of}%
}

The construction relies on the following simple observation.
\begin{lemma}\phantomsection\label{u-exists}
For a~collection of~hyperbolic maps $F_j:(\mathbb{C}, 0)\rightarrow (\mathbb{C}, 0)$, $F_j'(0)\neq 1$, there exists an~arbitrarily thick strip
\begin{equation}
U=\Set{\zeta |\Re \zeta >R, n_{-}<\Im \zeta <n_{+}},\quad n_{+}-n_{-}>C,
\phantomsection
\label{u-thick}
\end{equation}
such that $U$ does not overlap its images under ${\tilde{F}}_j$.
\end{lemma}
\begin{proof}
Recall that ${\tilde{F}}_j(\zeta )=k_j\zeta +b_j+o(1)$, see \hyperref[inzetachart]{Lemma 1}. For a~map~$F_j$ with $k_j\in \mathbb{R}$, the affine term $\zeta \mapsto k_j\zeta +b_j$ of~${\tilde{F}}$ has invariant horizontal line $\Im \zeta =y_j≔\frac{\Im b_j}{1-k_j}$, and for $|\Im \zeta -y_j|>\frac{C}{|k-1|}$ we have $|\Im \zeta -\Im (k\zeta +b)|>C$. Consider a~strip $U$ such that $|\Im \zeta -y_j|>\frac{C}{|k_j-1|}$ whenever $k_j\in \mathbb{R}$. Clearly, for $R$ large enough, all maps~${\tilde{F}}_j$ will be close enough to their respective affine terms, hence ${\tilde{F}}_j(U)\cap U=\varnothing $. Finally, we enlarge $R$ so that the assertion is satisfied for the maps ${\tilde{F}}_j$ with $k_j\notin \mathbb{R}$.
\end{proof}

Let $C_{1}$ be the estimate on $|\Im T|$ from \hyperref[f-j-head]{Lemma 7}; put $C_{2}=\max(C_{1}, |\Re \mu _{1}|+2)$. Fix a~strip \eqref{u-thick}, $C=2C_{2}+|\Im \mu _{1}|$, such that
\begin{equation}
\left(f_i^{(h)}\right)'(0)\neq \left(f_j^{(h)}\right)'(0)\quad\Rightarrow \quad \left(f_i^{(h)}\right)^{-1}\circ \left(f_j^{(h)}\right)(U)\cap U=\varnothing .
\phantomsection
\label{u-hyperbolic}
\end{equation}
Recall that ${\tilde{g}}_{2}(\zeta )=\zeta -\mu _{1}+o(1)$. Hence there exists a~small disc $\Delta \subset U$ such that the distance between $\mathbb{C}\smallsetminus U$ and $\Delta \cup {\tilde{g}}_{2}(\Delta )$ is greater than $C_{2}$, and $|{\tilde{g}}_{2}(\zeta )-\zeta +\mu _{1}|<1$ for $\zeta \in \Delta $. Shrinking $\Delta $ if necessary, we may and will assume that $\forall \zeta \in \Delta $ we have $|{\tilde{g}}_{2}'(\zeta )|\neq 1$. If $|{\tilde{g}}_{2}'(\zeta )|<1$ in $\Delta $, then we put $\Delta ^{+}=\Delta $, $g_{2}^\pm =g_{2}$, otherwise we put $\Delta ^{+}={\tilde{g}}_{2}(\Delta )$ and $g_{2}^\pm =g_{2}^{-1}$. Then ${\tilde{g}}_{2}^\pm $ contracts in $\Delta ^{+}$, see \eqref{g2-contracts}.

Finally, since the distance between $\mathbb{C}\smallsetminus U$ and $\Delta \cup {\tilde{g}}_{2}(\Delta )$ is greater than $C_{2}$, \hyperref[f-j-head]{Lemma 7} implies that for $N$ large enough, for each head ${\tilde{F}}_j^{(h)}$ of ${\tilde{F}}_j$
\begin{equation}
\left(F_j^{(h)}\right)'(0)=1\quad\Rightarrow \quad
{\tilde{F}}_j^{(h)}(\Delta ^{+})\subset U,\qquad
\left|\left({\tilde{F}}_j^{(h)}\right)'(\zeta )\right|<\frac{1}{\sqrt{q}},\qquad
\left|\left({\tilde{F}}_j^{(t)}\right)'(\zeta )\right|<\frac{1}{\sqrt{q}}
\phantomsection
\label{u-parabolic}
\end{equation}
for $\zeta \in {\tilde{g}}_{2}^\pm (\Delta ^{+})$, where $q$ is given by \eqref{g2-contracts}.

\subsubsection{Proof of the assumptions%
  \label{proof-of-the-assumptions}%
}

Let us prove that for $N$ large enough, the compositions ${\tilde{f}}_j$ satisfy the assumptions listed in the \hyperref[plan-of-the-proof-of-main-theorem]{plan of the proof}. For assumptions \hyperref[inclusion]{\textbf{inclusion}} and \hyperref[covering]{\textbf{covering}}, this immediately follows from \hyperref[f-j-head]{Lemma 7} and the definition of $T_j$.

Let us prove that \eqref{u-hyperbolic} and \eqref{u-parabolic} imply the \hyperref[contraction]{\textbf{contraction}} property. Consider a~composition of the form ${\tilde{f}}_i^{(t)}\circ {\tilde{f}}_j^{(h)}$. Recall that ${\tilde{f}}_i$ and ${\tilde{f}}_j$ are compositions of the commutators ${\tilde{g}}_{1}^{\pm 1}$ and ${\tilde{g}}_{2}^{\pm 1}$. Therefore, we can rewrite ${\tilde{f}}_i^{(h)}$, ${\tilde{f}}_j^{(h)}$ and the corresponding tails as
\begin{align*}
{\tilde{f}}_i^{(h)}&=\left({\tilde{g}}_k^{\pm 1}\right)^{(h)}\circ {\tilde{f}}_i^{(ph)}&{\tilde{f}}_i^{(t)}&={\tilde{f}}_i^{(pt)}\circ \left({\tilde{g}}_k^{\pm 1}\right)^{(t)},\\
{\tilde{f}}_j^{(h)}&= \left({\tilde{g}}_l^{\pm 1}\right)^{(h)}\circ {\tilde{f}}_j^{(ph)},&{\tilde{f}}_j^{(t)}&={\tilde{f}}_j^{(pt)}\circ \left({\tilde{g}}_k^{\pm 1}\right)^{(t)},
\end{align*}
where
\begin{itemize}

\item $f_i^{(ph)}$, $f_i^{(pt)}$, $f_j^{(ph)}$, $f_j^{(pt)}$ are compositions of $g_{1}^{\pm 1}$ and $g_{2}^{\pm 1}$;

\item $\set{k, l}\subset \set{1, 2}$;

\item $\left(g_k^{\pm 1}\right)^{(h)}$ and $\left(g_l^{\pm 1}\right)^{(h)}$ may be empty but may not coincide with $g_k^{\pm 1}$ or $g_l^{\pm 1}$.

\end{itemize}

If the maps~$\left(g_k^{\pm 1}\right)^{(h)}$ and~$\left(g_l^{\pm 1}\right)^{(h)}$ have different multipliers, then \eqref{u-hyperbolic} and \eqref{u-parabolic} imply that $(f_j^{(h)})^{-1}\circ f_i^{(h)}(\Delta ^{+})\cap \Delta ^{+}=\varnothing $.

Next, suppose that $\left(g_k^{\pm 1}\right)^{(h)}$ and~$\left(g_l^{\pm 1}\right)^{(h)}$ have equal multipliers. It is easy to check that our assumption on~$\mu _{1}$, $\mu _{2}$ implies that in this case $\left(g_k^{\pm 1}\right)^{(t)}\circ \left(g_l^{\pm 1}\right)^{(h)}$ is one of the maps $\id$, $g_k^{\pm 1}$, $g_l^{\pm 1}$. In the first case, we just eliminate the middle part from
\begin{equation*}
f_i^{(t)}\circ f_j^{(h)}={\tilde{f}}_i^{(pt)}\circ \left[\left({\tilde{g}}_k^{\pm 1}\right)^{(t)}\circ \left({\tilde{g}}_l^{\pm 1}\right)^{(h)}\right]\circ {\tilde{f}}_j^{(ph)},
\end{equation*}
and in the two latter cases we can regard the middle part either as a~part of~$f_i^{(t)}$, or as a part of $f_j^{(h)}$. Hence we can assume that both $f_i^{(t)}$ and $f_j^{(h)}$ are parabolic maps.

Finally, due to \eqref{u-parabolic}, for parabolic $f_i^{(t)}$ and $f_j^{(h)}$ we have
\begin{equation*}
\left|\left(f_i^{(t)}\circ f_j^{(h)}\right)'(\zeta )\right|<q\times \frac{1}{\sqrt{q}}\times \frac{1}{\sqrt{q}}=1.
\end{equation*}
Hence, the maps $f_j$ satisfy the \hyperref[contraction]{\textbf{contraction}} requirement.

\section{Construction of limit cycles%
  \label{construction-of-limit-cycles}%
}

Consider a polynomial foliation $\mathcal{F}\in \mathcal{A}_n'$. Suppose that there exist domains $\Delta ^{-}\subset \Delta ^{+}$, a chart $\zeta $ and a tuple of monodromy maps $f_j$ that satisfy the assumptions listed in the \hyperref[plan-of-the-proof-of-main-theorem]{plan of the proof}. In this section we shall show that such foliation satisfies the assertions of \hyperref[main-theorem]{Main Theorem}. The proof is based on the following simple observation.
\begin{lemma}\phantomsection\label{independent-cycles}
Suppose that a collection of limit cycles $c_j$ satisfies the following:
\setcounter{listcnt0}{0}
\begin{list}{\alph{listcnt0})}
{
\usecounter{listcnt0}
\setlength{\rightmargin}{\leftmargin}
}

\item all cycles $c_j$ are simple, i.e., have no self-intersections;

\item their multipliers $\mu (c_j)$ satisfy $0<|\mu (c_j)|<|\mu (c_{1})\cdots \mu (c_{j-1})|$;

\item $c_i\cap c_j=\varnothing $ for $i\neq j$.
\end{list}

Then these cycles are homologically independent.
\end{lemma}
\begin{proof}
Since all these cycles are simple and do not intersect each other, a possible dependency has the form $\pm [c_{i_{1}}] \pm [c_{i_{2}}]\pm \ldots \pm [c_{i_s}]=0$, $i_{1}<i_{2}<\ldots <i_s$. However such dependence implies the equality on multipliers, $\mu (c_{i_{1}})^{\pm 1}\mu (c_{i_{2}})^{\pm 1}\ldots \mu (c_{i_{s-1}})^{\pm 1}=\mu (c_{i_s})$, which is impossible due to the inequality
\begin{equation*}
|\mu (c_{i_s})|<|\mu (c_{1})\ldots \mu (c_{i_s-1})|\leqslant |\mu (c_{i_{1}})\ldots \mu (c_{i_{s-1})}|\leqslant |\mu (c_{i_{1}})^{\pm 1} \mu (c_{i_{2}})^{\pm 1}\ldots \mu (c_{i_{s-1}})^{\pm 1}|.
\end{equation*}\end{proof}

\DUadmonition[note]{
\DUtitle[note]{Note}

In earlier papers \cite{Il78,SRO98} the authors used similar arguments, but they estimated $\int _{c_j} x\,dy-y\,dx$ instead of multipliers. This led to much more complicated computations.
}

As we mentioned above, \hyperref[hutchinson]{Lemma 5} enables us to construct cycles with arbitrarily small multipliers, but these cycles may be neither simple, nor disjoint. The following two lemmas fill these gaps.
\begin{lemma}\phantomsection\label{avoid-finite}
Let $D\subset \mathbb{R}^{n}$ be a closed disc, $g_{1},g_{2}:D\rightarrow D$ be two injective continuous maps such that $g_{1}(D)\cap g_{2}(D)=\varnothing $, $\Sigma \subset D$ be a finite subset. Then for $m$ large enough there exists a periodic orbit
\begin{equation}
p_{0}, p_{1}, \ldots , p_m=p_{0}, p_{i+1}\in \set{g_{1}(p_i), g_{2}(p_i)}
\phantomsection
\label{periodic}
\end{equation}
that never meets $\Sigma $.
\end{lemma}
\begin{lemma}\phantomsection\label{simple-subcycle}
Given
\begin{itemize}

\item an open subset $U\Subset \Delta ^{-}$;

\item two maps $g_{1}=f_{i_{1}}\circ \ldots \circ f_{i_s}$, $g_{2}=f_{j_{1}}\circ \ldots \circ f_{j_r}$, $g_i:\Delta ^{+}\rightarrow U$ with disjoint images;

\item a finite set $\Sigma \subset S$;

\item a positive number $\varepsilon $,

\end{itemize}

there exists a finite set $\Sigma '\subset \Delta ^{+}$ such that the following holds. Suppose that a~periodic orbit \eqref{periodic} never visits $\Sigma '$. Since $p_{0}$ is a fixed point of some monodromy map, it corresponds to a~cycle~$c$. Let $c'$ be its simple subcycle. Then $c'$ visits $U$, never visits $\Sigma $, and the modulus of its multiplier is less than~$\varepsilon $.
\end{lemma}

Let us deduce \hyperref[main-theorem]{Main Theorem} from these two lemmas.
\begin{proof}[{Proof of Main Theorem}]
Fix a~sequence of points $x_k$ in the interior of $\Delta ^{-}$ dense in $\Delta ^{-}$. Let $U_k$ be the intersection of $\Delta ^{-}$ with the $(1/k)$-neighborhood of $x_k$. Now we construct the sequence $c_j$ by induction. Suppose that simple homologically independent cycles $c_{1}, c_{2},\ldots , c_{k-1}$ are already constructed and have multipliers $\mu (c_j)$, $|\mu (c_j)|<1$. Put $\Sigma =\bigcup _j c_j\cap S$, $\varepsilon  = |\mu (c_{1})\mu (c_{2})\cdots \mu (c_{k-1})|$.

Take two disjoint domains $V_{1}, V_{2}\subset U_k$. Due to \hyperref[hutchinson]{Lemma 5}, there exist two contracting compositions $g_{1}:\Delta ^{+}\rightarrow V_{1}$ and $g_{2}:\Delta ^{+}\rightarrow V_{2}$. According to the previous two lemmas, there exists a simple cycle $c_k$ with multiplier less than $\varepsilon $ that intersects $U_k$ but does not visit $\Sigma $. Note that this cycle, as well as all previous ones, projects to the curve of the form $\gamma _{l_{1}}\gamma _{l_{2}}\ldots  \gamma _{l_r}$ on the infinite line, and $\gamma _i\cap \gamma _j=\set{O}$ for $i\neq j$. Thus if $c_i\cap c_k\neq \varnothing $ for some $i<k$, then $c_k\cap c_i\cap S\neq \varnothing $, hence $c_k\cap \Sigma \neq \varnothing $ which contradicts the choice of $c_k$. Due to \hyperref[independent-cycles]{Lemma 9}, the cycles $c_{1}, \ldots , c_k$ are homologically independent.
\end{proof}

Now let us prove the lemmas formulated above.
\begin{proof}[{Proof of Lemma 10}]
Fix a large number $m$. Due to Brouwer Theorem, for each word $w=w_{1}\ldots w_m$, $w_i\in \set{1, 2}$, the corresponding map
\begin{equation*}
g_w=g_{w_{1}}\circ \ldots \circ g_{w_m}:D\rightarrow D
\end{equation*}
has a fixed point. Our goal is to find a word $w$ such that the corresponding periodic orbit will never visit $\Sigma $. Since $g_{1}(D)\cap g_{2}(D)=\varnothing $, the images of $2^m$ maps $g_w$, $|w|=m$, are pairwise disjoint. Hence, given a point $p\in \Sigma $, there is at most one word $w$ such that $g_w(p)=p$. Their cyclic shifts are the only words $w$ such that the corresponding periodic orbit visits $\Sigma $, thus there are at most $|\Sigma |\cdot m$ of them. Clearly, for $m$ large enough we have $|\Sigma |\cdot m<2^m$, hence there exists a periodic orbit of length~$m$ that never visits~$\Sigma $.
\end{proof}
\begin{proof}[{Proof of Lemma 11}]
Consider a~composition $g_w=g_{w_{1}}\circ \ldots \circ g_{w_s}=\mathbf{M}_{i_{1}}\circ \ldots \circ \mathbf{M}_{i_k}$, the corresponding periodic orbit \eqref{periodic} and the corresponding limit cycle $c$. If \eqref{periodic} does not visit the finite set $\Sigma _{1}=\bigcup_{g_i^{(h)}}\left(g_i^{(h)}\right)^{-1}(\Sigma )$, then $c\cap \Sigma =\varnothing $.

Subcycles of $c$ correspond to representations $g_w=g^{(t)}\circ g^{(m)}\circ g^{(h)}$ with non-empty $g^{(m)}$ such that $g^{(h)}(p_{0})$ is a fixed point of~$g^{(m)}$.

Let us prove that sufficiently long compositions~$g^{(m)}$ correspond to cycles $c'$ that satisfy the assertions of the lemma, and fixed points of~“short” compositions can be avoided by avoiding a~finite set $\Sigma _{2}$.

If $g^{(m)}$ is not a subcomposition of one of $f_j$, then it can be represented in the form
\begin{equation}
g^{(m)}=f_{s_{1}}^{(h)}\circ f_{s_{2}}\circ \ldots \circ f_{s_{k-1}}\circ f_{s_k}^{(t)}.
\phantomsection
\label{gm}
\end{equation}
Recall that $f_{s_k}^{(t)} \circ f_{s_{1}}^{(h)}$ contracts  in the chart $\zeta $ due to \hyperref[contraction]{\textbf{contraction}} property, thus $(f_{s_{1}}^{(h)})^{-1} \circ g^{(m)} \circ f_{s_{1}}^{(h)}$ contracts, and we have
\begin{equation}
\mu (c') \leqslant \left(\max_{\zeta \in \Delta ^{+}}f_i'(\zeta )\right)^{k-2}.
\phantomsection
\label{multiplier}
\end{equation}
Let $\operatorname{len}(\cdot )$ be the length of a composition of $\mathbf{M}_j$. For sufficiently large  $L$,  $\operatorname{len}(g^{(m)})\geqslant L$ implies that the subcycle $c'$  corresponding to $g^{(m)}$ satisfies the assertions of the lemma.  Indeed, the multiplier of $c'$ can be made arbitrarily small due to \eqref{multiplier};
for any $L > \max (\operatorname{len} g_{1}, \operatorname{len} g_{2})$, the corresponding $c'$ visits $U$ because it contains a point of the form $g_{w_{i+1}}\circ \ldots \circ g_{w_s}(p_{0})\in U$.

Now, it is sufficient to avoid fixed points of~“short” compositions $g^{(m)}$, $\operatorname{len}(g^{(m)})<L$.
Let us prove that none of the compositions $g^{(m)}$ are identical in $g^{(h)}(\Delta ^{+})$. Suppose the contrary. We have eliminated all identical subcompositions from $f_j$, so $g^{(m)}$ cannot be a~subcomposition of~some~$f_j$. Thus $g^{(m)}$ has the form \eqref{gm}, and \eqref{multiplier} yields that $g^{(m)}$ contracts. Hence $g^{(m)}\neq \id$.

Therefore, each composition $g^{(m)}$ has only finitely many fixed points in $g_{w_j}^{(h)}(\Delta ^{+})$, where $g_{w_j}^{(h)}(\Delta ^{+})$ is defined by $g^{(h)}=g_{w_j}^{(h)}\circ g_{w_{j+1}}\circ \ldots \circ g_{w_s}$. In order to guarantee that a subcycle $c'$ corresponds to a long composition $g^{(m)}$, $\operatorname{len} g^{(m)}>L$, it is sufficient to require that the periodic orbit \eqref{periodic} avoids the finite set
\begin{equation*}
\Sigma _{2}=\set{(g_{w_j}^{(h)})^{-1}\Fix g^{(m)}|\operatorname{len}(g^{(m)})<L}.
\end{equation*}
The required exceptional set is $\Sigma '=\Sigma _{1}\cup \Sigma _{2}$.
\end{proof}

\section{Acknowledgements%
  \label{acknowledgements}%
}

We proved these results and wrote the first version of this article during our 5-months visit to Mexico (Mexico City, then Cuernavaca). We are very grateful to UNAM (Mexico) and HSE (Moscow) for supporting this visit. Our deep thanks to Laura Ortiz Bobadilla, Ernesto Rosales-González and Alberto Verjovsky, for invitation to Mexico and for fruitful discussions.

We are thankful to Arsenij Shcherbakov for useful discussions about technical details of \cite{SRO98}. We are also grateful to Yulij Ilyashenko for permanent encouragement, and to Victor Kleptsyn for interesting discussions.


\begin{thebibliography}{Shch06}

\bibitem[BLL97]{BLL97}{
M. Belliart, I. Liousse, F. Loray \emph{Sur l’existence de points fixes attractifs pour les sous-groupes de $Aut(\mathbb{C}, 0)$} // C. r. Acad. sci. Paris. Sér. 1. 1997. V. 324, N 4. P. 443–446.
}

\bibitem[H81]{H81}{
J. Hutchinson, \emph{Fractals and self-similarity}, Indiana Univ. Math. J., 30:5 (1981), 713–747.
}

\bibitem[Il78]{Il78}{
Yu. S. Ilyashenko,  \emph{Topology of phase portraits of analytic differential equations on a complex projective plane}, Trudy Sem. Petrovsk., Vol. 4 (1978), pp. 83–136
}

\bibitem[IYa07]{IYa07}{
Yu. Ilyashenko, S. Yakovenko, \emph{Lectures on analytic differential equations}, Graduate Studies in Mathematics, Vol. 86, AMS, 2007
}

\bibitem[KhV62]{KhV62}{
M. G. Khudaĭ-Verenov, \emph{A property of the solutions of a differential equation} (in Russian), Math. USSR Sb., Vol. 56(98) (1962), pp. 301–308
}

\bibitem[N94]{N94}{
I. Nakai \emph{Separatrices for non-solvable dynamics on $(\mathbb{C}, 0)$} // Ann. Inst. Fourier. 1994. V. 44, N 2. P. 569–599.
}

\bibitem[Shch06]{Shch06}{
A. Shcherbakov, \emph{Dynamics of local groups of conformal mappings and generic properties of differential equations on $\mathbb{C}^{2}$}, Proceedings of the Steklov Institute of Mathematics, Vol. 254, pp. 103–120
}

\bibitem[Shch84]{Shch84}{
A. Shcherbakov, \emph{Topological and analytical conjugacy of non-commutative groups of germs of conformal mappings}, Trudy Sem. Petrovsk., Vol. 10 (1984), pp. 170–196
}

\bibitem[SRO98]{SRO98}{
A. Shcherbakov, E. Rosales-González, L. Ortiz-Bobadilla, \emph{Countable set of limit cycles for the equation $dw/dz=P_n(z,w)/Q_n(z,w)$}, J. Dynam. Control Systems, Vol. 4 (1998), No. 4, pp. 539–581
}

\end{thebibliography}
\end{document}